\documentclass[11pt]{article}
\usepackage{amsmath,amsfonts,amsthm,amscd,amssymb,graphicx}
\numberwithin{equation}{section}

\newtheorem{theorem}{Theorem}[section]

\newtheorem{lemma}[theorem]{Lemma}

\newtheorem{remark}[theorem]{Remark}
\newtheorem{definition}[theorem]{Definition}

\newcommand{\RR}{\mathbb{R}}
\newcommand{\cO}{\mathcal{O}}

\newcommand{\cE}{\mathcal{E}}

\newcommand{\e}{{\varepsilon}}

\def\beq{\begin{equation}}
\def\eeq{\end{equation}}
\def\bb1{{1\!\!1}}

\def\cR{\mathcal{R}}

\def\w{{\omega}}

\def\Div{\mathrm{div~}}

\def\intO{\int_\Omega}
\def\intdO{\int_{\partial\Omega}}

\begin{document}

\title{Remarks on the inviscid limit for the compressible flows}

\author{Claude Bardos\footnotemark[1]  \and Toan T. Nguyen\footnotemark[2] 
}

\renewcommand{\thefootnote}{\fnsymbol{footnote}}

\footnotetext[1]{Laboratoire J.-L. Lions \& Universit\'e Denis Diderot, BP187, 75252 Paris Cedex 05, France. Email: claude.bardos@gmail.com}

\footnotetext[2]{Department of Mathematics, Pennsylvania State University, State College, PA 16802, USA. Email: nguyen@math.psu.edu. TN's research was supported in part by the NSF under grant DMS-1405728.
}

\date{ }

\maketitle

\begin{abstract}
We establish various criteria, which are known in the incompressible case, for the validity of the inviscid limit for the compressible Navier-Stokes flows considered in a general domain $\Omega$ in $\RR^n$ with or without a boundary. In the presence of a boundary, a generalized Navier boundary condition for velocity is assumed, which in particular by convention includes the classical no-slip boundary conditions. In this general setting we extend the Kato criteria and show the  convergence to a solution which is dissipative "up to the boundary". In the case of smooth solutions, the convergence is obtained in the relative energy norm. 
\end{abstract}

$\quad${\bf Subject classifications.}  Primary 76D05, 76B99, 76D99.
\vskip5pt
\rightline{\em   For  H. Beira\~o da Veiga as a token of gratefulness and friendship$\quad \quad$}
\tableofcontents

\section{Introduction}
We are interested in the inviscid limit problem for compressible flows. Precisely, we consider the following compressible model \cite{Feir1, Feir2} of Navier-Stokes equations which consist of the two fundamental principles of conservation of mass and momentum: 
\begin{equation}\label{NS}
\begin{aligned}
 \rho_t + \nabla \cdot (\rho u) & = 0\\
(\rho u)_t + \nabla\cdot (\rho u \otimes u) +  \nabla p(\rho,\theta) &=  \nabla \cdot \epsilon\sigma(\nabla u)
\end{aligned}\end{equation}
in which the density $\rho \ge 0$, the velocity $u \in \RR^n$, the pressure $p = p(\rho)$ satisfying the $\gamma$-pressure law: $p = a_0 \rho^\gamma$, with $a_0>0, \gamma>1$, and the viscous stress tensor $\sigma(\nabla u)$ defined by \begin{equation}\label{def-stress}
\begin{aligned}
\sigma(\nabla u) &= \mu \Big[ (\nabla u + (\nabla u)^t) - \frac 23 (\nabla \cdot u ) I\Big ] + \eta (\nabla \cdot u) I 
\end{aligned}\end{equation}
with positive constants $\eta, \mu$. Here, $\e$ is a small positive parameter. The Navier-Stokes equations are considered in a domain $\Omega\subset \RR^n$, $n=2,3$. In the presence of a boundary, we assume the following generalized Navier boundary conditions for velocity: 
\begin{equation}\label{Navier-BC}u\cdot n = 0, \qquad  \epsilon \sigma (\nabla u) n \cdot \tau + \lambda_\e (x) u \cdot \tau= 0 \qquad \mbox{on} \quad \partial\Omega\end{equation}
with $\lambda_\e(x)\ge 0$ and $n, \tau$ being the outward normal and tangent vectors at $x$ on $\partial\Omega$. Here, by convention, we include $\lambda_\e = \infty$, in which case the above condition reduces to the classical no-slip boundary condition: 
\begin{equation}\label{no-slip} u_{\vert_{\partial\Omega}} =0.\end{equation} 
We shall consider both boundary conditions throughout the paper. We note that since $u \cdot n =0$ on the boundary, there is no boundary condition needed for the density function $\rho$. 

{\em We are interested in the problem when $\e \to 0$.} Naturally, one would expect in the limit to recover the compressible Euler equations: 
\begin{equation}\label{Euler}
\begin{aligned}
 \bar \rho_t + \nabla \cdot (\bar \rho \bar u) &= 0\\
  (\bar \rho \bar u)_t + \nabla \cdot (\bar \rho \bar u \otimes \bar u) + \nabla p(\bar \rho, \bar \theta) &= 0
  \end{aligned}\end{equation}
with the boundary condition:
$$ \bar u\cdot n =0, \qquad \mbox{on}\quad \partial \Omega .$$

Like the incompressible case when the density and temperature are homogenous ($\rho =1$), the inviscid limit problem is a very delicate issue, precisely
due to the appearance of boundary layer flows, compensating the discrepancy in the boundary conditions for the Navier Stokes and Euler equations (see, for instance, \cite{Asano, Const1, Vicol, WE, Grenier, Kato,Kelliher1, Kelliher, Mas, Mekawa, MT, SC1, SC2, Temam, Wang} and the references therein). In this paper, we shall establish several criteria, which are known in the incompressible case, for the inviscid limit to hold. These criteria can also be naturally extended to the compressible flows of Navier-Stokes-Fourier equations \cite{FN,FN1,N}, at least in the case when the temperature satisfies the zero Neumann boundary condition. 
Over the years  the contributions of Hugo to Mathematical Theory of Fluid Mechanics have been instrumental. In particular  with his articles  \cite{H1, H2}  he has shown recent interest in boundary layers. Therefore we are happy to dedicate to our friend H. Beira\~o da Veiga this present paper, on the occasion of his 70th birthday.

\subsection{Definitions}

We shall introduce the definition of weak solutions of Navier-Stokes and Euler equations that are used throughout the paper.  First, we recall that for smooth solutions to the Navier-Stokes equations, by multiplying the momentum equation by $u$ and integrating the result, we easily obtain the energy balance: 
\begin{equation}\label{id-EE}\frac{d}{dt} \cE(t) + \e \intO \sigma(\nabla u) : \nabla u\; dx  + \intdO \lambda_\e(x) |u|^2\; d \sigma = 0\end{equation}
in which $\cE(t)$ denotes the total energy defined by 
\begin{equation}\label{def-E} \cE(t) : = \intO \Big[ \rho \frac{|u|^2}{2} +H(\rho) \Big]\;dx , \qquad H(\rho): =  \frac{a_0 \rho^\gamma}{\gamma-1}.\end{equation}
 Here, $A:B$ denotes the tensor product between two matrices $A= (a_{jk})$, and $B = (b_{jk})$; precisely, $A:B = \sum_{j,k} a_{jk}b_{jk}$. In particular,  the energy identity yields a priori bound for the total energy $\cE(t)$ as well as the total dissipation thanks to the inequality: there is a positive constant $\theta_0$ so that 
 \begin{equation}\label{grad-bound} \intO \sigma(\nabla u) : \nabla u   \ge \theta_0 \intO |\nabla u|^2. \end{equation}
In the case of no-slip boundary conditions or in the domain with no boundary, the boundary term in the energy balance \eqref{id-EE} vanishes. 

Following Feireisl at al. \cite{Feir1, Feir2}, we introduce the following notion of weak solutions: 

 \begin{definition}[Finite energy weak solutions to Navier-Stokes] \label{def-NS} Let $(\rho_0,u_0)$ be some initial data so that $\rho_0\ge 0, \rho_0\in L^\gamma(\Omega), \rho_0 u_0^2 \in L^1(\Omega)$ and let $T $ be a fixed positive time.  The pair of functions $(\rho,u)$ is called a finite energy weak solution to Navier-Stokes if the following hold:

\begin{itemize}

\item $\rho \ge 0, \rho \in L^\infty (0,T; L^\gamma), u \in L^2(0,T; H^1(\Omega))$.

\item the Navier-Stokes equations in \eqref{NS} are satisfied in the usual distributional sense. 

\item the total energy $\cE(t)$ is locally integrable on $(0,T)$ and there holds the energy inequality: 
\begin{equation}\label{finiteEE} \cE(t)
 + \e \int_0^t \intO \sigma(\nabla u) : \nabla u  +\int_0^t  \intdO \lambda_\e(x) |u|^2\; d \sigma \le \cE(0).\end{equation}
\end{itemize}

\end{definition}


\begin{remark}{\em  Feireisl at al have shown in \cite{Feir2} that such a finite energy weak solution to Navier-Stokes exists globally in time, with the $\gamma$-pressure law of $\gamma >3/2$. 
} 
\end{remark}

Feireisl at al. \cite{Feir1} also introduces a notion of weak suitable solutions based on relative entropy and energy inequalities. There, they start with the notion of relative entropy function following Dafermos \cite{Dafermos} (also see \cite{Germain, Feir1,Feir2}): 
\begin{equation}\label{def-relE}H(\rho ; r) = H(\rho) - H(r) - H'(r) (\rho - r) \end{equation}
for all $\rho, r\ge 0$, in which $H(\rho) = \frac{a_0 \rho^\gamma}{\gamma-1}$ as defined in \eqref{def-E}, and the relative energy function associating with the solutions $(\rho, u)$ to the Navier-Stokes equations
\begin{equation}\label{def-relEE} 
\cE(\rho,u; r,w)(t): = \intO \Big( \frac 12 \rho |u - w|^2 + H(\rho; r) \Big) (t)  ,
\end{equation}
for all smooth test functions $(r,w)$. Let us note that since the function $H(\rho)$ is convex in $\{ \rho >0\}$, the function $H(\rho; r)$ can serve as a distance function between $\rho$ and $r$, and hence $\cE(\rho, u; r,w)$ can be used to measure the stability of the solutions $(\rho, u)$ as compared to test functions $(r,w)$. For instance, for any $r$ in a compact set in $(0,\infty)$, there holds
\begin{equation}\label{equiv-bdH}\begin{aligned}
 H(\rho;r) \quad \approx \quad |\rho - r|^2 \chi_{\{|\rho - r|\le 1\}} + |\rho - r|^\gamma \chi_{\{ |\rho - r|\ge 1\}} , \qquad \forall \rho\ge 0,
 \end{aligned}\end{equation}
in the sense that $H(\rho;r)$ gives an upper and lower bound in term of the right-hand side quantity. The bounds might depend on $r$. 

Next, let $(\rho,u)$ satisfy the Navier-Stokes equations in the distributional sense. That is, $(\rho,u)$ solves 
\begin{equation}\label{weak-1}\intO \rho r(t) \; dx = \intO \rho_0 r(0)\; dx + \int_0^t \intO (\rho \partial_t r + \rho u \cdot \nabla r ) \; dxdt \end{equation}
and 
\begin{equation}\label{weak-2}\begin{aligned} 
\intO \rho u w(t)\; dx + \int_0^t \Big[ \intO \rho u \cdot \partial_t w &+ \rho u \otimes u : \nabla w + p(\rho) \Div w  - \e \sigma(\nabla u) : \nabla w\Big] \; dxdt 
\\
&= \intO \rho_0 u_0 w(0)\; dx - \int_0^t \intdO \lambda_\e(x) u \cdot w ,
\end{aligned}\end{equation}
for any smooth test functions $(r,w)$ defined on $[0,T]\times \bar \Omega$ so that $r$ is bounded above and below away from zero, and 
$ w\cdot n = 0$ on $\partial \Omega.$ We remark that for such a test function, there holds the uniform equivalent bound \eqref{equiv-bdH}. 

Then, a direct calculation (\cite{Feir1, Feir2,Sueur}) yields 
\begin{equation}\label{rel-ineq} \begin{aligned}
\cE(\rho, u; r,w)(t)  &+ \e \int_0^t \intO \sigma(\nabla u) : \nabla u  +\int_0^t  \intdO \lambda_\e(x) |u|^2\; d \sigma
\\&\le \cE(\rho, u; r,w)(0) + \int_0^t \cR (\rho, u; r,w),
\end{aligned}\end{equation}
for almost every $t $ in $[0,T]$, in which 
\begin{equation}\label{def-R} \begin{aligned}
\cR(\rho, u; r,w): &= \intO \Big[ \rho (\partial_t + u \cdot \nabla ) w \cdot (w-u)  + \e \sigma(\nabla u) : \nabla w\Big] + \intdO \lambda_\e(x)u \cdot w\; d\sigma
\\
&\quad  + \intO \Big( (r-\rho) \partial_t H'(r) + (rw - \rho u)\cdot \nabla H'(r) \Big)
\\&\quad 
- \intO \Big (\rho (H'(\rho) - H'(r) ) - H(\rho; r)\Big)\Div w.
\end{aligned}\end{equation}

\begin{definition}[Suitable solutions to Navier-Stokes]\label{def-suitablesolutions} The pair $(\rho,u)$ is called a suitable solution to Navier-Stokes equations if $(\rho,u)$ is a renormalized weak solution in the sense of DiPerna-Lions \cite{DiPL} and the relative energy inequality \eqref{rel-ineq} holds for any smooth test functions $(r,w)$ defined on $[0,T]\times \bar \Omega$ so that $r$ is bounded above and below away from zero, and 
$ w\cdot n = 0$ on $\partial \Omega.$ 
\end{definition}

This motivates us to introduce the notion of dissipative weak solutions to Euler equations, following DiPerna and Lions \cite{DiPL}. Indeed, in the case of Euler when $\epsilon =0$, the relative energy inequality reads
\begin{equation}\label{rel-ineq-Euler} \begin{aligned}
\cE(\bar \rho, \bar u; r,w)(t)  &\le \cE(\rho_0, u_0; r,w)(0) + \int_0^t \cR^0 (\bar \rho, \bar u; r,w),
\end{aligned}\end{equation}
in which $\cR^0(\bar \rho, \bar u; r,w)$ is defined as in \eqref{def-R} with $\epsilon =0$ and no boundary term. In addition, if we assume further that the smooth test functions $(r,w)$ solve
\begin{equation}\label{EErw-def}
\begin{aligned}
r_t + \nabla \cdot (rw) & = 0\\
 (\partial_t + w \cdot \nabla) w + \nabla H'(r) &= E(r,w).
\end{aligned}\end{equation} for some residual $E(r,w)$, then a direct calculation and a straightforward estimate (for details, see Section \ref{sec-relestimates} and inequality \eqref{Euler-relenergy}) immediately yields 
$$ \begin{aligned}
\cR^0(\bar \rho, \bar u; r,w) &\le \intO \Big[\rho E(r,w)\cdot (w-\bar u) + c_0(r)\|\Div w\|_{L^\infty(\Omega)}  H(\bar \rho; r) \Big] \; dx 
\end{aligned}$$ 
for some positive constant $c_0(r)$ that depends only on the upper and lower bound of $r$ as in the estimate \eqref{equiv-bdH}.  Clearly, $$  \intO H(\bar \rho; r)(t) \; dx \le \cE(\bar \rho, \bar u; r,w)(t) .$$
Hence, the standard Gronwall's inequality applied to \eqref{rel-ineq-Euler}, together with the above estimates, yields 
\begin{equation}\label{Euler-relenergy-10}\begin{aligned}
\cE(\bar \rho, \bar u; r,w)(t)  &\le \cE(\bar \rho, \bar u; r,w)(0) e^{c_0(r)\int_0^t \|\Div w(\tau)\|_{L^\infty(\Omega)} \; d\tau}\\
&\quad + \int_0^t e^{c_0(r)\int_s^t \|\Div w(\tau)\|_{L^\infty(\Omega)} \; d\tau}\intO \rho E(r,w)\cdot (w-\bar u) \; dxds.
\end{aligned}
\end{equation}

Let us now introduce the notion of dissipative solutions to Euler; see its analogue in the incompressible case by  Bardos and Titi \cite{BT}, Definition 3.2 and especially Definition 4.1 taking into account of boundary effects.

 \begin{definition}[Dissipative solutions to Euler] \label{def-Euler} The pair $(\bar \rho, \bar u)$ is a dissipative solution of Euler equations if and only if $(\bar \rho, \bar u)$ satisfies the relative energy inequality \eqref{Euler-relenergy-10} for all smooth test functions $(r,w)$ defined on $[0,T]\times \bar \Omega$ so that $r$ is bounded above and below away from zero, 
$ w\cdot n = 0$ on $\partial \Omega$, and $(r,w)$ solves \eqref{EErw-def}.
\end{definition}

%
%


\begin{remark}\label{rem-dissipative}{\em If the Euler equations admit a smooth solution $(r,w)$, then the residual term $E(r,w) =0$ in \eqref{EErw-def} and hence the relative energy inequality \eqref{Euler-relenergy-10} yields the uniqueness of dissipative solutions, within class of weak solutions of the same initial data as those of the smooth solution.  }
 \end{remark}

\begin{remark}{\em
In case of no boundary, weak solutions of Euler that satisfy the energy inequality are dissipative solutions. This is no long true in case with boundaries; see a counterexample that weak solutions satisfying the energy inequality are not dissipative solutions, due to a boundary; see counterexamples in \cite{BSW}. }
\end{remark}

\subsection{Main results}

Our main results are as follows: 

\begin{theorem}[Absence of boundaries]\label{theo-nobdry} Let $(\rho_\e,u_\e)$ be any finite energy weak solution to Navier-Stokes in domain $\Omega$ without a boundary. Then, any weak limit $(\bar \rho, \bar u)$ of $(\rho_\e, u_\e)$ in the sense: 
\begin{equation}\label{def-weakconv}\begin{aligned}
 \rho_\e &\rightharpoonup \bar \rho, \qquad \mbox{weakly in }\quad L^\infty(0,T; L^\gamma (\Omega))
 \\
\rho_\e u_\e^2 &\rightharpoonup \bar \rho\bar u^2, \qquad \mbox{weakly in }\quad L^\infty(0,T; L^2(\Omega))
\end{aligned}\end{equation}
as $\e\to0$, is a dissipative solution to the Euler equations.  \end{theorem}

\begin{theorem}[Presence of a boundary]\label{theo-bdry} Assume the generalized Navier boundary condition \eqref{Navier-BC} holds; in particular, we allow the case of no-slip boundary conditions \eqref{no-slip}. Let $(\rho_\e,u_\e)$ be any finite energy weak solution to Navier-Stokes and let $(\bar \rho, \bar u )$ be a weak limit of $(\rho_\e, u_\e)$ in the sense of \eqref{def-weakconv}. Then,  $(\bar \rho, \bar u)$ is a dissipative solution to Euler equations in the sense of Definition \ref{def-Euler} if any of one of the following conditions holds: 

\begin{enumerate}

\item[i.] (Bardos-Titi's criterium) $\e \sigma(\nabla u_\e)n \cdot \tau \to 0$ or equivalently $\epsilon (\omega_\e \times n)\cdot \tau \to 0$, as $\e\to 0$, in the sense of distribution in $(0,T)\times \partial\Omega$. Here, $\omega_\e = \nabla \times u_\e$. 

\item[ii.] (Kato-Sueur's criterium) The sequence $(\rho_\e, u_\e) $ satisfies the estimate: 
\begin{equation}\label{Kato} \lim_{\e\to 0} \int_0^T \int_{\Omega \cap \{ d(x,\partial \Omega) \le \e\}} \Big[ H(\rho_\e) + \e \frac{\rho_\e |u_\e|^2}{d(x,\partial\Omega)^2} +  \e |\nabla u_\e|^2 \Big]\; dx dt =0.\end{equation}

\item[iii.] (Constatin-Kukavica-Vicol's criterium) $\bar u\cdot \tau \ge 0$ almost everywhere on $\partial \Omega$, $\rho_\e$ is uniformly bounded, and the vorticity $\omega_\e=\nabla \times u_\e$ satisfies 
\begin{equation}\label{CKV-condition}  \e (\omega_\e \times n)\cdot \tau  \ge -M_\e(t)\qquad  \mbox{with } \qquad \lim_{\e\to0} \int_0^T M_\e(t) \; dt \le 0.\end{equation}

\end{enumerate} 

\end{theorem}

\begin{remark} {\em In the case that the density function $\rho_\e$ is uniformly bounded, by a use of the Hardy's inequality, the condition \eqref{Kato} reduces to the original Kato's condition as in the incompressible case, namely
\begin{equation}\label{Kato1} \lim_{\e\to 0} \int_0^T \int_{\Omega \cap \{ d(x,\partial \Omega) \le \e\}} \e |\nabla u_\e|^2\; dx dt =0.\end{equation}}
\end{remark}

\begin{remark}{\em
Sueur \cite{Sueur} proves that given a strong solution $(\bar \rho, \bar u)$ to Euler in $C^{1+\alpha} ((0,T)\times \Omega)$ with $\bar\rho$ being bounded above and away from zero, there is a sequence of finite energy weak solutions to Navier-Stokes that converges to the Euler solution in the relative energy norm in both cases: 
\begin{enumerate}

\item[i.] No-slip boundary condition under the Kato's condition;  

\item[ii.] Navier boundary condition with $\lambda_\e \to 0$.

\end{enumerate}
Theorem \ref{theo-bdry} recovers Sueur's results in both of these cases, thanks to the weak-strong uniqueness property of the dissipative solutions; see Remark \ref{rem-dissipative}.  
}
\end{remark}





\section{Proof of the main theorems}

\subsection{Stress-free condition}
Let us write the stress-free boundary condition in term of vorticity. 
\begin{lemma}\label{lem-stress} Let the stress tensor $\sigma(\nabla u)$ be defined as in \eqref{def-stress}, $\omega = \nabla \times u$, and let $u\cdot n=0$ on the boundary $\partial\Omega$. There holds
$$ \sigma(\nabla u) n \cdot \tau =  \mu  (\omega \times n)\cdot \tau - \kappa u\cdot \tau$$
on $\partial \Omega$, in which $\kappa: = 2\mu  (\tau \cdot \nabla) n \cdot\tau$ with $n$ and $\tau$ being normal and tangent vectors to the boundary $\partial \Omega$.
\end{lemma} 
\begin{proof} Let us work in $\RR^3$. By convention, $\nabla u$ is the matrix with column being $\partial_{x_j} u$, $u \in \RR^3$, for each column $j = 1,2,3$. A direct calculation gives 
$$ (\nabla u - (\nabla u)^t)) n \cdot \tau = (\omega \times n)\cdot \tau .$$
Next, we compute 
$$ (\nabla u)^T n \cdot \tau = \sum_{k,j} \tau_k\partial_k u_j n_j  = \tau \cdot \nabla (u\cdot n) - (\tau \cdot \nabla) n \cdot u.$$
By using the assumption that $u \cdot n=0$ on the boundary and $\tau$ is tangent to the boundary, $ \tau \cdot \nabla (u\cdot n) =0$ on $\partial \Omega$. 
Next, by definition, we have
$$
\begin{aligned}
\sigma(\nabla u)n \cdot \tau &= \mu (\nabla u + (\nabla u)^t) n \cdot \tau  + (\eta - \frac 23 \mu)(\Div u) n\cdot \tau
\end{aligned}$$ 
which completes the proof of the lemma, by the above calculations and the fact that $ n\cdot \tau =0$.
\end{proof}

\subsection{Relative energy estimates}\label{sec-relestimates}
In this section, let us derive some basic relative energy estimates. We recall the remainder term in the relative energy inequality \eqref{rel-ineq} defined by 
\begin{equation}\label{def-RNS} \begin{aligned}
\cR(\rho, u; r,w): &= \intO \Big[ \rho (\partial_t + u \cdot \nabla ) w \cdot (w-u)  + \e \sigma(\nabla u) : \nabla w\Big] + \intdO \lambda_\e(x)u \cdot w\; d\sigma
\\
&\quad  + \intO \Big( (r-\rho) \partial_t H'(r) + (rw - \rho u)\cdot \nabla H'(r) \Big)
\\&\quad 
- \intO \Big (\rho (H'(\rho) - H'(r) ) - H(\rho; r)\Big)\Div w
\end{aligned}\end{equation}
where $(r,w)$ are smooth test functions. If we assume further that the pair $(r,w)$ solves
\begin{equation}\label{EErw}
\begin{aligned}
r_t + \nabla \cdot (rw) & = 0\\
 (\partial_t + w \cdot \nabla) w + \nabla H'(r) &= E(r,w)
\end{aligned}\end{equation}
for some residual $E(r,w)$, then a direct calculation immediately yields 
$$ \begin{aligned}
\cR(\rho, u; r,w) &= \intO \Big[\rho E(r,w)\cdot (w-u) + \e \sigma(\nabla u) : \nabla w\Big] + \intdO \lambda_\e(x)u \cdot w\; d\sigma
\\&\quad 
- \intO \Big (\rho (H'(\rho) - H'(r)) - r (\rho-r) H''(r) - H(\rho; r)\Big)\Div w.
\end{aligned}$$ 

%

\begin{lemma} Let $H(\rho;r)$ be defined as in \eqref{def-relE}. Let $r$ be arbitrary in a compact set of $(0,\infty)$. There holds 
$$\rho (H'(\rho) - H'(r)) - r (\rho-r) H''(r) \quad \approx\quad  H(\rho;r).$$
for all $\rho \ge 0$.
\end{lemma}
\begin{proof} Indeed, let us write 
$$\begin{aligned}
\rho (H'(\rho) - H'(r)) - r (\rho-r) H''(r)  &= (\rho-r) ( H'(\rho) - H'(r)) +  (\rho-r)^2 H''(r)
\\
&\quad  + r[ H'(\rho) - H'(r) - H''(r) (\rho -r)], 
\end{aligned}$$
which is clearly of order $|\rho -r|^2$ when $|\rho -r|\le 1$, and hence of order of $H(\rho;r)$. Now, when $|\rho -r|\ge 1$, we have
$$\rho (H'(\rho) - H'(r)) - r (\rho-r) H''(r)  \le c_0(r) \left\{ \begin{aligned} |\rho - r| \le |\rho-r|^\gamma, &\qquad \mbox{when} \quad \rho \le r\\
\rho^\gamma - r^\gamma \le |\rho-r|^\gamma, &\qquad \mbox{when} \quad \rho \ge r.
\end{aligned}\right.$$
This proves the lemma. \end{proof}

Using the lemma, the relative energy inequality \eqref{rel-ineq} reduces to 
\begin{equation}\label{NS-relenergy-pf}\begin{aligned}
\cE(\rho, &u; r,w)(t)  + \e \int_0^t \intO \sigma(\nabla u) : \nabla u  +\int_0^t  \intdO \lambda_\e(x) |u|^2\; d \sigma
\\&\le \cE(\rho, u; r,w)(0) + \int_0^t\intO \Big[\rho E(r,w)\cdot (w-u) + \e \sigma(\nabla u) : \nabla w\Big] 
\\&\quad + c_0(r)\|\Div w\|_{L^\infty(\Omega)} \int_0^t \intO H(\rho; r)\; dxds+ \intdO \lambda_\e(x)u \cdot w\; d\sigma,
\end{aligned}
\end{equation}
for all smooth test functions $(r,w)$ that solve \eqref{EErw}. In particular, the same calculation with $\epsilon =0$ and with no boundary term yields 
\begin{equation}\label{Euler-relenergy}\begin{aligned}
\cE(\bar \rho, \bar u; r,w)(t)  \le& \cE(\bar \rho, \bar u; r,w)(0) + \int_0^t\intO \rho E(r,w)\cdot (w-\bar u) \; dxds
\\& +  c_0(r)\|\Div w\|_{L^\infty(\Omega)} \int_0^t \intO H(\bar \rho; r) \; dxds
\end{aligned}
\end{equation}
for $(\bar \rho, \bar u)$ solving the Euler equations in the sense of \eqref{weak-1} and \eqref{weak-2}. 
\subsection{Absence of boundaries: proof of Theorem \ref{theo-nobdry}}
Let $(\rho, u)$ be a finite energy weak solution to Navier-Stokes, and let $(\bar \rho, \bar u)$ be a weak limit of $(\rho, u)$ as $\e \to 0$ in the sense as in Theorem \ref{theo-nobdry}. 
We remark that the weak convergences in the theorem immediately yield $\rho u\rightharpoonup \bar \rho \bar u$ weakly in $L^\infty(0,T; L^{\frac{2\gamma}{\gamma+1}}(\Omega))$. We shall show that $(\bar \rho, \bar u)$ is a dissipative solution to Euler in the sense of Definition \ref{def-Euler}. Indeed, let $(r,w)$ be any smooth test functions defined on $[0,T]\times \RR^n$ so that $r$ is bounded above and below away from zero, and $(r,w)$ solves 
\begin{equation}\label{eqs-rw}\begin{aligned}
r_t + \nabla \cdot (rw) & = 0\\
 (\partial_t + w \cdot \nabla) w + \nabla H'(r) &= E(r,w)
\end{aligned}
\end{equation}for some residual $E(r,w)$. The starting point is the relative energy inequality \eqref{NS-relenergy-pf}. We note that there exists a positive constant $\theta_0$ so that 
\begin{equation}\label{H1-bound}\intO \sigma(\nabla u) : \nabla u   \ge \theta_0 \intO |\nabla u|^2\end{equation}
and so the energy inequality reads
\begin{equation}\label{NS-relenergy-1}\begin{aligned}
\cE(\rho, &u; r,w)(t)  + \e\theta_0 \int_0^t \intO 
|\nabla u|^2 
\\&\le \cE(\rho, u; r,w)(0) + \int_0^t\intO \Big[\rho E(r,w)\cdot (w-u) + \e \sigma(\nabla u) : \nabla w\Big] 
\\&\quad + c_0(r)\|\Div w\|_{L^\infty(\Omega)} \int_0^t \cE(\rho, u; r,w)(\tau)\; d\tau 
\end{aligned}
\end{equation}
in which the integral 
\begin{equation}\label{Y-bound} \e \intO \sigma(\nabla u) : \nabla w \le \frac {\e \theta_0} 2 \int |\nabla u|^2 + \e C_0 \intO |\nabla w|^2\end{equation}
has the first term on the right absorbed into the left hand-side of \eqref{rel-ineq}, whereas the second term converges to zero as $\e\to 0$. Hence, \begin{equation}\label{NS-relenergy-2}\begin{aligned}
\cE(\rho, u; r,w)(t) &\le \cE(\rho, u; r,w)(0) + \int_0^t\intO \Big[\rho E(r,w)\cdot (w-u) +C_0 \e |\nabla w|^2\Big] 
\\&\quad + c_0(r)\|\Div w\|_{L^\infty(\Omega)} \int_0^t \cE(\rho, u; r,w)(\tau)\; d\tau 
\end{aligned}
\end{equation}
which by the Gronwall's inequality then yields 
\begin{equation}\label{NS-relenergy-3}\begin{aligned}
\cE(\rho,& u; r,w)(t)  \le \cE(\rho_0, u_0; r,w)(0) e^{c_0(r)\int_0^t \|\Div w(\tau)\|_{L^\infty(\Omega)} \; d\tau}\\
&\quad + 
\int_0^t e^{c_0(r)\int_s^t \|\Div w(\tau)\|_{L^\infty(\Omega)} \; d\tau}\intO \Big[\rho E(r,w)\cdot (w-u) +C_0 \e |\nabla w|^2\Big]  \; dxds.
\end{aligned}
\end{equation}
We now let $\e\to 0$ in the above inequality. We have  
$$ \cE(\bar \rho, \bar u; r,w)(t)  \le \liminf_{\e\to 0} \cE(\rho, u; r,w)(t) $$ 
and by the fact that $\rho (w-u)$ converges weakly in $L^\infty(0,t; L^{\frac{2\gamma}{\gamma+1}}(\Omega) )$, 
$$ \begin{aligned}
\lim_{\e\to0} \int_0^t &e^{c_0(r) \int_s^t \|\Div w(\tau)\|_{L^\infty(\Omega)} \; d\tau}\intO\Big[\rho E(r,w)\cdot (w-u) +C_0 \e |\nabla w|^2\Big]  \; dxds 
\\& = \int_0^t e^{c_0(r)\int_s^t \|\Div w(\tau)\|_{L^\infty(\Omega)} \; d\tau}\intO \bar \rho E(r,w)\cdot (w-\bar u)  \; dxds.
\end{aligned}$$
Hence, we have obtained the inequality in the limit:
\begin{equation}\label{Euler-relenergy-3}\begin{aligned}
\cE(\bar \rho,& \bar u; r,w)(t)  \le \cE(\rho_0, u_0; r,w)(0) e^{c_0(r)\int_0^t \|\Div w(\tau)\|_{L^\infty(\Omega)} \; d\tau}\\
&\quad + 
\int_0^t e^{c_0(r)\int_s^t \|\Div w(\tau)\|_{L^\infty(\Omega)} \; d\tau}\intO \bar \rho E(r,w)\cdot (w-\bar u) 
\; dxds.
\end{aligned}
\end{equation}
This proves that $(\bar \rho, \bar u)$ is a dissipative solution to Euler, which gives Theorem \ref{theo-nobdry}. 

\subsection{Presence of a boundary: proof of Theorem \ref{theo-bdry}}

\subsubsection{Proof of Bardos-Titi's criterium} Let us first prove (i) of Theorem \ref{theo-bdry}. Similarly as in the above case when no boundary is present, let $(r,w)$ be any smooth test functions defined on $[0,T]\times \bar \Omega$ so that  
$ w\cdot n = 0$ on $\partial \Omega,$ and $(r,w)$ solves \eqref{eqs-rw}. Then, there holds the relative energy inequality \eqref{NS-relenergy-pf}, together with the estimate \eqref{Y-bound}:
$$\begin{aligned}
\cE(\rho, &u; r,w)(t)  + \e \theta_0\int_0^t \intO | \nabla u|^2  -\int_0^t  \intdO  \e \sigma(\nabla u)n \cdot u\; d \sigma
\\&\le \cE(\rho, u; r,w)(0) + \int_0^t\intO \Big[\rho E(r,w)\cdot (w-u) + C_0\e | \nabla w|^2\Big] 
\\&\quad + c_0(r)\|\Div w\|_{L^\infty(\Omega)} \int_0^t \cE(\rho, u; r,w)(\tau)\; d\tau - \int_0^t\intdO \e \sigma(\nabla u)n \cdot w \; d\sigma .
\end{aligned}
$$
The assumption in $(i)$ is made precisely so that $\e \sigma(\nabla u)n \cdot \tau\to 0$, as $\e\to 0$, weakly. We note that this is equivalent to the assumption that $\e (\omega \times n)\cdot \tau \to 0$ by Lemma \ref{lem-stress}. Hence, the last boundary integral term vanishes in the limit, or rather 
$$ \begin{aligned}
\lim_{\e\to0} \int_0^t \intdO  \e \sigma(\nabla u)n  \cdot we^{c_0(r) \int_s^t \|\Div w(\tau)\|_{L^\infty(\Omega)} \; d\tau} \; d\sigma =0.
\end{aligned}$$
Whereas, the boundary term appearing on the left of the above inequality is either nonnegative, if the Navier boundary condition $-\e \sigma(\nabla u)n \cdot \tau = \lambda_\epsilon u \cdot \tau$ is assumed, with $\lambda_\e \ge 0$, or vanishes if the no-slip boundary condition $u_\e =0$ on $\partial\Omega$ is assumed. Hence, in the limit of $\e\to0$, we obtain the same inequality as in \eqref{Euler-relenergy-3}, which proves that $(\bar \rho, \bar u)$ is a dissipative solution to Euler equations.

\subsubsection{Proof of Kato-Sueur's criterium} Next, we prove the second statement (ii). We shall show that 
\begin{equation}\label{weak-bdry} \lim_{\e\to0} \int_0^T \intdO \e \sigma(\nabla u)n \cdot w \; d\sigma ds =0\end{equation}
for all smooth test functions $w$ so that $w\cdot n =0$ on the boundary $\partial\Omega$. This and the statement (i) would then yield (ii). To do so, let $w$ be the test function. We then construct a Kato fake layer $w_\e$, following Kato \cite{Kato}, or in fact, Sueur  \cite[Section 2.2]{Sueur}) as follows: 
\begin{equation}\label{fake-layer} w_\e := w \chi(\frac{d(x,\partial\Omega)}{c_0\e}) ,\end{equation} 
for arbitrary positive constant $c_0$, in which $\chi(\cdot)$ is a smooth cut-off function so that $\chi(0) =1$ and $\chi(z) =0$ for $z\ge 1$. It follows that  so that $w_\e  = w$ on the boundary $\partial\Omega\times (0,T)$ and has its support contained in the domain: $ \Gamma_\e \times (0,T)$ with $ \Gamma_\e = \{ x\in \Omega~:~ d(x,\partial\Omega) \le c_0\e\}.$
Furthermore, $w_\e$ satisfies 
\begin{equation}\label{bound-we}\| \Div w_\e\|_{L^\infty} +  \|\partial_t w_\e\|_{L^\infty} +  \| \e  \nabla w_\e\|_{L^\infty}\le C.\end{equation}
Let us now use $w_\e$ as a test function in the weak formulation \eqref{weak-2}, yielding 
\begin{equation}\label{weak-3}\begin{aligned} 
 \int_0^T &\intdO \e \sigma(\nabla u)n \cdot w_\e \; d\sigma dt
 \\&=  \int_0^T  \intO \Big[\rho u \cdot \partial_t w_\e + \rho u \otimes u : \nabla w_\e + p(\rho) \Div w_\e  - \e \sigma(\nabla u) : \nabla w_\e\Big] \; dxdt.
\end{aligned}\end{equation}
Since $w_\e = w$ on the boundary $\partial\Omega$, in order to prove \eqref{weak-bdry}, it suffices to show that each term on the right converges to zero as $\e\to 0$. We treat term by term. First, we have 
$$ \e \int_0^T \intO | \sigma(\nabla u) : \nabla w_\e | \le \sqrt \e \| \nabla u\|_{L^2(\Gamma_\e \times (0,T))} \| \sqrt \e\nabla w_\e\|_{L^2(\Gamma_\e \times (0,T))} $$
which converges to zero as $\e\to0$, since $ \sqrt \e \| \nabla u\|_{L^2(\Gamma_\e \times (0,T))} \to 0$ by \eqref{Kato} and $\e\nabla w_\e$ is bounded in $L^\infty$ and hence $\sqrt \e \nabla w_\e$ bounded in $L^2(\Gamma_\e\times (0,T))$, thanks to the fact that the Lebesgue measure of $\Gamma_\e \times (0,T)$ is of order $\e$. Next, we estimate 
$$  \int_0^T\intO \rho u \cdot \partial_t w_\e \; dxds \le \| \rho u \|_{L^{\frac{2\gamma}{\gamma+1}}(\Omega\times (0,T))} \| \partial_t w_\e\|_{L^{\frac{2\gamma}{\gamma-1}}(\Gamma_\e \times (0,T))}  \to 0,$$ 
as $\e\to 0$, since $ \rho u$ and  $\partial_t w_\e$ are uniformly bounded in $L^{\frac{2\gamma}{\gamma+1}}$ and $L^\infty$, respectively, and the Lebesgue measure of $\Gamma_\e \times (0,T)$ tends to zero. Similarly, we also have 
$$  \int_0^T \intO |p(\rho) \Div w_\e|\; dxdt \le C \int_0^T \int_{\Gamma_\e} H(\rho)\; dxdt$$ 
which converges to zero by the assumption. Finally, we estimate 
$$ \begin{aligned} \int_0^T \intO  |\rho u \otimes u : \nabla w_\e | \le \int_0^T \int_{\Gamma_\e} \frac{\rho |u|^2}{d(x,\partial\Omega)^2} \e^2| \nabla w_\e| \le C \e  \int_0^T \int_{\Gamma_\e} \frac{\rho |u|^2}{d(x,\partial\Omega)^2},
 \end{aligned}  $$
which again tends to zero as $\e\to 0$ by the assumption. Combining these into \eqref{weak-3} proves that 
$$ \int_0^T \intdO \e \sigma(\nabla u)n \cdot w \; d\sigma dt =  \int_0^T \intdO \e \sigma(\nabla u)n \cdot w_\e \; d\sigma dt \to 0$$
as $\e\to0$, for any smooth test function $w$ with $w\cdot n =0$ on $\partial\Omega$. That is, $\e \sigma(\nabla u)n\cdot \tau \to 0$ in the distributional sense, and hence (ii) follows directly from (i).

\subsubsection{Proof of Constantin-Kukavica-Vicol's criterum}
Finally, let us prove (iii). By a view of the proof of (i), it suffices to prove that 
$$\lim_{\e \to 0} \Big[ - \int_0^T\intdO \e \sigma(\nabla u)n \cdot w \; d\sigma  - \frac{\e\theta_0}{2} \int_0^T \intO |\nabla u|^2\Big] \le 0,$$ for all smooth test functions $w\ge 0.$ Let $w_\e$ be the Kato fake layer defined as in \eqref{fake-layer},
in which we can assume further that the cut-off function satisfies $\chi' \le 0$. Then, as in \eqref{weak-3}, we have
\begin{equation}\label{weak-4}\begin{aligned} 
-\int_0^T& \intdO  \e \sigma(\nabla u)n \cdot w \; d\sigma =  -\int_0^T \intdO \e \sigma(\nabla u)n \cdot w_\e \; d\sigma dt
 \\&=  -\int_0^T  \intO \Big[\rho u \cdot \partial_t w_\e + \rho u \otimes u : \nabla w_\e + p(\rho) \Div w_\e  - \e \sigma(\nabla u) : \nabla w_\e\Big] \; dxdt.
\end{aligned}\end{equation}
Clearly, the first and the third integrals converge to zero, thanks to the boundedness of $\rho$ in $\Gamma_\e\times [0,T]$ and \eqref{fake-layer}, whereas the second is treated as
$$ \begin{aligned} \int_0^T \intO  |\rho u \otimes u : \nabla w_\e | \le \frac{C \e}{c_0}  \int_0^T \int_{\Gamma_\e} \frac{|u|^2}{d(x,\partial\Omega)^2} \le \frac{C \e}{c_0}  \int_0^T \int_{\Gamma_\e} |\nabla u|^2,
 \end{aligned}  $$
in which we used the Hardy's inequality and the fact that $\e|\nabla \w_\e | \le C/c_0$ for $c_0$ as in the definition of $\Gamma_\e$. We can take $c_0$ sufficiently small so that $C/c_0 \le \theta_0/2$.  

Finally, we deal with the integral involving the stress tensor. We estimate the integral following \cite{Vicol}. Let us decompose $w_\e = n w^n_\e + \tau w^\tau_\e$, in which $n,\tau$ are local orthogonal basis vectors in the neighborhood $\Gamma_\e$ so that $n$ is in the direction of the distance from $x$ to the boundary $\partial\Omega$. 
By the construction and the fact that $w^n_\e =0$ on the boundary, we have
$$\| \tau \cdot \nabla w_\e\|_\infty +  \| n\cdot \nabla w^n_\e\|_\infty + \| \e n \cdot \nabla w^\tau_\e  \|_\infty \le C.$$
Hence, we have 
\begin{equation}\label{dn-we} \nabla w_\e = \tau^tn (n\cdot \nabla w_\e^\tau) + \cO(1),\end{equation}
in which $\cO(1)$ means uniformly bounded by a constant $C$. Let us now treat the integral 
$$\int_0^T  \intO\e \sigma(\nabla u) : \nabla w_\e \; dxdt = \int_0^T  \intO\e \Big[ D(u): \nabla w_\e + (\eta - \frac23 \mu)(\Div u )(\Div w_\e)\Big]\; dxdt $$
in which $D(u): = \mu  (\nabla u + \nabla u^t)$. The last integral in the above identity is bounded by 
$$  \e  (\eta - \frac23 \mu) \| \nabla u\|_{L^2(\Gamma_\e \times (0,T))} \| \Div w_\e\|_{L^2(\Gamma_\e \times (0,T))} \le C \e^{3/2}  \| \nabla u\|_{L^2(\Gamma_\e \times (0,T))} $$
which converges to zero, since $\sqrt \e  \| \nabla u\|_{L^2(\Gamma_\e \times (0,T))}$ is bounded by the energy inequality \eqref{finiteEE}. To treat the first integral, we use \eqref{dn-we} to get
\begin{equation}\label{est-Du-we}\begin{aligned}
\int_0^T  \intO\e D(u): \nabla w_\e &\le \int_0^T  \intO\e D(u): \tau^t n (n\cdot \nabla w_\e^\tau) + C\e \int_0^T\int_{\Gamma_\e} |\nabla u| 
\\&\le \int_0^T  \intO\e D(u) n \cdot \tau (n\cdot \nabla w_\e^\tau) + C\e^{3/2} \| \nabla u\|_{L^2(\Gamma_\e \times (0,T))}. 
\end{aligned}\end{equation}
By the computation in the proof of Lemma \ref{lem-stress}, we get 
$$ D(u) n \cdot \tau = \mu (\omega \times n)\cdot \tau  + 2\mu \tau \cdot \nabla (u\cdot n) - 2\mu (\tau \cdot \nabla) n \cdot u.$$
We put this expression into \eqref{est-Du-we}. We estimate each of the integral as follows: 
$$\begin{aligned}
\Big| \int_0^T  \intO\e (\tau \cdot \nabla) n \cdot u (n\cdot \nabla w_\e^\tau)\Big| \le C \int_0^T \int_{\Gamma_\e }|u | \le C \e \int_0^T\int_{\Gamma_\e} |\nabla u| \le \cO(\e) \end{aligned}$$ 
and by integration by parts twice, upon noting that $\tau \cdot \nabla (u\cdot n) = 0$ on the boundary, we obtain 
$$\begin{aligned}
 \int_0^T  \intO\e \tau \cdot \nabla (u\cdot n) (n\cdot \nabla w_\e^\tau)  
 &=  \int_0^T  \intO\e n \cdot \nabla (u\cdot n) (\tau\cdot \nabla w_\e^\tau)  + \cO(\e)  \int_0^T\int_{\Gamma_\e} |\nabla u|
\\
&\le C\e  \int_0^T\int_{\Gamma_\e} |\nabla u| \le \cO(\e).
 \end{aligned}$$ 
Finally, we note that by definition \eqref{fake-layer}, the directional derivative 
$$n\cdot \nabla w_\e^\tau  = \frac{1}{c_0 \e} \chi'  (\frac{d(x,\partial\Omega)}{c_0\e}) w + \chi (\frac{d(x,\partial\Omega)}{c_0\e}) n \cdot \nabla w.$$ 
in which $\chi' w \le 0$ on $\Gamma_\e$. Next, by the assumption on the vorticity \eqref{CKV-condition} we get 
$$\begin{aligned}
\int_0^T  \intO \mu \e (\omega \times n)\cdot \tau (n\cdot \nabla w_\e^\tau)  
&\le - \int_0^T \int_{\Gamma_\e} \frac{M_\e(t)}{c_0 \e} \chi'  (\frac{d(x,\partial\Omega)}{c_0\e}) w  + C\e \int_0^T \int_{\Gamma_\e} |\omega|
\\
&\le C\int_0^TM_\e(t)\; dt  + \cO(\e).
 \end{aligned}$$ 

This proves (iii) and hence the main theorem: Theorem \ref{theo-bdry}.

\section{Conclusion and remark on the  Navier-Stokes-Fourier systems}

One of the main concern in this contribution is the existence $/$ non existence of generation of vorticity (equivalent to the anomalous dissipation of energy) in a layer of the order of $\epsilon$ in the zero viscosity limit.

What appears is that this issue remains very similar when  the compressible equations are considered instead of the incompressible and this is in spite of the fact that (in particular  in $2d$) the questions of regularity, stability and uniqueness are much more complicated for the compressible equation than for the incompressible case.  In fact as observed in \cite{BGP} the situation is also similar for the macroscopic limit in the incompressible scaling of solutions of the Boltzmann equation.

Along the same line one may conjecture and we leave it for further works that the situation would be similar for the full Navier-Stokes-Fourier equations. As long as no stringent boundary layer is generated at the level of the temperature equation (this would be achieved by using the Neuman boundary condition for this equation).

More precisely one may  start with  the following compressible model \cite{FN,FN1,N} of Navier-Stokes equations which consist of the three fundamental principles of conservation of mass, momentum, and entropy: 
\begin{equation}\label{NSF2}
\begin{aligned}
 \rho_t + \nabla \cdot (\rho u) & = 0\\
(\rho u)_t + \nabla\cdot (\rho u \otimes u) +  \nabla p(\rho,\theta) &=  \nabla \cdot \epsilon\sigma(\nabla u)
\\
(\rho s(\rho,\theta))_t + \nabla \cdot (\rho s(\rho,\theta) u) -  \nabla \cdot \frac{\kappa(\nabla \theta)}{\theta} & =  r(\nabla u, \nabla \theta)
\end{aligned}\end{equation}
in which the density $\rho \ge 0$, the velocity $u \in \RR^n$, the temperature $\theta > 0$, the pressure $p = p(\rho,\theta)$, and the specific entropy $s = s(\rho,\theta)$, together with the heat flux $-\kappa\nabla \theta$,  the viscous stress tensor $\sigma(\nabla u)$, and the entropy production rate $ r(\nabla u, \nabla \theta)$, respectively defined by $$
\begin{aligned}
\sigma(\nabla u) &= \mu \Big[ (\nabla u + (\nabla u)^t) - \frac 23 (\nabla \cdot u ) I\Big ] + \eta (\nabla \cdot u) I 
\\
r(\nabla u, \nabla \theta)& =  \frac{1}{\theta} \Big(\epsilon \sigma (\nabla u) : \nabla u - \frac{\kappa (\nabla \theta) \cdot \nabla \theta}{\theta} \Big)
\end{aligned}$$ 
with positive constants $\eta, \mu, \kappa$.  In the presence of a boundary, we assume  as above the following generalized Navier boundary conditions for velocity: 
\begin{equation}\label{Navier-BCF}u\cdot n = 0, \qquad  \epsilon \sigma (\nabla u) n \cdot \tau + \lambda_\e (x) u \cdot \tau= 0 \qquad \mbox{on} \quad \partial\Omega\end{equation}
with $\lambda_\e(x)\ge 0$ and $n, \tau$ being the outward normal and tangent vectors at $x$ on $\partial\Omega$  (which also with $\lambda_\e = \infty$  reduces to the classical no-slip boundary condition).


We note that since $u \cdot n =0$ on the boundary, there is no boundary condition needed for the density function $\rho$. Whereas, we assume the zero heat transfer, or Neumann, boundary condition:
\begin{equation}\label{Temp-BC} \nabla \theta \cdot n = 0,  \qquad \mbox{on} \quad \partial\Omega.\end{equation}
This seems to be essential to avoid the generation of boundary  layer that would (at the level of dissipation of energy) interact with the viscous boundary layer for velocity.

Then the formal limit of the above problem  (for  $\epsilon$ and $\kappa$ going to zero ) is   the compressible Euler 
\begin{equation}\label{EulerF2}
\begin{aligned}
 &\bar \rho_t + \nabla \cdot (\bar \rho \bar u) = 0\,,\\
 & (\bar \rho \bar u)_t + \nabla \cdot (\bar \rho \bar u \otimes \bar u) + \nabla p(\bar \rho, \bar \theta)= 0\,,
  \\
  &(\bar \rho s(\bar \rho,\bar \theta))_t + \nabla \cdot (\bar \rho s(\bar \rho,\bar \theta) u)  =0\,.
  \end{aligned}\end{equation}
with the unique  boundary condition
$ \bar u\cdot n =0\,.$

Then one may conjecture that the statement and the proof of the theorem \ref{theo-bdry} can be adapted for the discussion of the convergence of solutions  of (\ref{NSF2} ) to solutions of (\ref{EulerF2}) with no major modification.


\end{document}